\documentclass[12pt]{article}
\usepackage{array,amsmath}
\usepackage{amssymb}
\usepackage{graphicx,subfigure}

\makeatletter
 \usepackage{fullpage} 
 \usepackage{amsthm}

\usepackage{fullpage}

\makeatother

\begin{document}

\newtheorem{thm}{Theorem}[section]
\newtheorem{lem}[thm]{Lemma}
\newtheorem{prop}[thm]{Proposition}
\newtheorem{cor}[thm]{Corollary}
\newtheorem{con}[thm]{Conjecture}
\newtheorem{claim}[thm]{Claim}
\newtheorem{obs}[thm]{Observation}
\newtheorem{defn}[thm]{Definition}
\newtheorem{example}[thm]{Example}
\newcommand{\di}{\displaystyle}
\def\dfc{\mathrm{def}}
\def\cF{{\cal F}}
\def\cH{{\cal H}}
\def\cK{{\cal K}}
\def\cC{{\cal C}}
\def\cA{{\cal A}}
\def\cB{{\cal B}}
\def\cP{{\cal P}}
\def\ap{\alpha'}
\def\Frk{F_k^{2r+1}}
\def\nul{\varnothing} 
\def\st{\colon\,}   
\def\MAP#1#2#3{#1\colon\,#2\to#3}
\def\VEC#1#2#3{#1_{#2},\ldots,#1_{#3}}
\def\VECOP#1#2#3#4{#1_{#2}#4\cdots #4 #1_{#3}}
\def\SE#1#2#3{\sum_{#1=#2}^{#3}}  \def\SGE#1#2{\sum_{#1\ge#2}}
\def\PE#1#2#3{\prod_{#1=#2}^{#3}} \def\PGE#1#2{\prod_{#1\ge#2}}
\def\UE#1#2#3{\bigcup_{#1=#2}^{#3}}
\def\FR#1#2{\frac{#1}{#2}}
\def\FL#1{\left\lfloor{#1}\right\rfloor} 
\def\CL#1{\left\lceil{#1}\right\rceil}

\title{Hamiltonicity in Connected Regular Graphs}
\author{
Daniel W. Cranston\thanks{Department of Mathematics and Applied Mathematics,
Virginia Commonwealth University, Richmond, VA, 23284,
\texttt{dcranston@vcu.edu}}
\and
Suil O\thanks{Department of Mathematics, The College of William and Mary,
Williamsburg, VA, 23185, \texttt{so@wm.edu.}}
}

\maketitle

\begin{abstract}
In 1980, Jackson proved that every 2-connected $k$-regular graph 
with at most $3k$ vertices is Hamiltonian. This result has been extended in
several papers. In this note, we determine the minimum number of vertices in a
connected $k$-regular graph that is not Hamiltonian, and we also solve the
analogous problem for Hamiltonian paths.  Further, we characterize the smallest
connected $k$-regular graphs without a Hamiltonian cycle.

\end{abstract}

\section {Introduction}

In 1980, Jackson~\cite{J1} gave a sufficient condition on the number of vertices
in a 2-connected $k$-regular graph for it to be Hamiltonian.
A graph $G$ is {\it $k$-connected} if it has more than $k$ vertices
and every subgraph obtained by deleting fewer than $k$ vertices
is connected. A graph $G$ is {\it Hamiltonian} if it contains a spanning cycle.
For terminology and notation not defined here, we use \cite{W}.

\begin{thm}{\rm (Jackson~\cite{J1})}\label{J1}
Every 2-connected $k$-regular graph on at most $3k$ vertices
is Hamiltonian.
\end{thm}

Theorem~\cite{J1} has been extended in several papers.
Hilbig~\cite{H} extended it to graphs on $3k+3$ vertices with two exceptions. 
Let $P$ denote the Petersen graph, and let $P'$ denote the graph obtained from
$P$ by replacing one vertex $v$ of $P$ by the complete graph $K_3$ and
making each vertex of the $K_3$ adjacent to a distinct neighbor of $v$.

\begin{thm}{\rm (Hilbig~\cite{H})}\label{H}
If G is a 2-connected $k$-regular graph on at most $3k+3$ vertices
and $G \notin \{P, P'\}$, then $G$ is Hamiltonian.
\end{thm} 

In Section~\ref{results}, we show that every connected $k$-regular graph on at
most $2k+2$ vertices has no cut-vertex, which implies by Theorem~\ref{J1} that
it is Hamiltonian.  In addition, we characterize connected $k$-regular graphs
on $2k+3$ vertices ($2k+4$ vertices when $k$ is odd) that are non-Hamiltonian.

A {\it Hamiltonian path} is a spanning path. We also solve the analogous
problem for Hamiltonian paths.  
Recall from Theorem~\ref{H} that every 2-connected $k$-regular graph $G$ on at
most $3k+3$ vertices is Hamiltonian, except for when $G\in\{P,P'\}$.
So to show that every connected $k$-regular graph on at most $3k+3$
vertices has a Hamiltonian path, it suffices to investigate $P$, $P'$, and
connected $k$-regular graphs with a cut-vertex.



\section{Maximum Number of Vertices for Hamiltonicity}
\label{results}

\begin{thm}\label{main1}
Every connected $k$-regular graph on at most $2k+2$ vertices is Hamiltonian.
Furthermore, we characterize connected k-regular graphs on $2k+3$ vertices (when
$k$ is even) and $2k+4$ vertices (when $k$ is odd) that are non-Hamiltonian.
\end{thm}
\begin{proof} Let $G$ be a connected $k$-regular graph on at most $2k+2$
vertices. 
By Theorem~\ref{J1}, it suffices to show that $G$ has no cut-vertex.  Assume to
the contrary that $G$ has a cut-vertex, $v$.  Now $G-v$ has at least two
components, say $O_1$ and $O_2$.  Let $H_i=G[V(O_i)\cup \{v\}]$ for $i\in [2]$.  Since
all vertices in $G$ have degree $k$ and each vertex in $H_i$ except $v$ has its
neighbors in $H_i$, the number of vertices in $H_i$ is at least $k+1$.  If
$|H_i|=k+1$, then $H_i = K_{k+1}$, which contradicts the fact that $v$ is a
cut-vertex.  Thus, each component of $G-v$ has at least $k+1$ vertices,
which implies that $G$ has at least $2k+3$ vertices.  This is a contradiction. 
(Note that by the degree sum formula, $|V(G)|$ is even if $k$ is odd.  Thus, if
$G$ has a cut-vertex and $k$ is odd, then $G$ has at least $2k+4$ vertices.)

Now we characterize the smallest connected $k$-regular graphs that are not
Hamiltonian; these show that the bound $2k+2$ in the theorem is optimal.  
(Our characterization relies on graphs first defined in~\cite{OW1}
and~\cite{OW2}.) As shown above, if a connected $k$-regular graph $G$ is
non-Hamiltonian, then $G$ has at least $2k+3$ vertices if $k$ is even,
and at least $2k+4$ vertices if $k$ is odd.  Specifically, $G$ must have a
cut-vertex, v (for otherwise, it is Hamiltonian by Theorem~\ref{J1}), and each
component of $G-v$ must have at least $k+1$ vertices.  When $k$ is even, the
degree sum formula shows that $v$  must have an even number of neighbors in each
component of $G-v$.  A similar argument works when $k$ is odd.  Thus, the
description below gives a complete characterization of such graphs.
We begin with the case when $k$ is even, and the case when $k$ is odd
is similar.

Let $k=2r$ for $r\ge2$.  
For $2 \le t \le 2r-2$ and $t$ even, let $F_{r,t}$
be a graph on $2r+2$ vertices with one vertex of degree $t$ and the remaining
$2r+1$ vertices of degree $2r$.  We can form such a graph from a
copy of $K_{2r+1}$ by deleting a matching on $t$ vertices, then adding a new
vertex adjacent to the $t$ endpoints of the matching.
Let $F'_{r,t}$ be a graph on $2r+1$ vertices with $t+1$ vertices of degree $2r$
and $2r-t$ vertices of degree $2r-1$.  We form such a graph from a copy
of $K_{2r+1}$ by deleting a matching on $2r-t$ vertices.

Let $\cF_r$ be the family of $2r$-regular graphs obtained from $F_{r,t}$ and
$F'_{r,t}$ by adding edges from the vertex of degree $t$ in $F_{r,t}$ to the
$2r-t$ vertices of degree $2r-1$ in $F'_{r,t}$.  Since each such graph contains
a cut-vertex, the family $\cF_r$ consists entirely of $2r$-regular graphs on
$2k+3$ vertices that are non-Hamiltonian.

Now let $k=2r+1$ for $r\ge 1$. For $2 \le t \le 2r$ and $t$ even, let $H_{r,t}$
be a graph on $2r+3$ vertices with one vertex of degree $t$ and the remaining
$2r+2$ vertices of degree $2r+1$.  As above, to form such a graph, begin with a
copy of $K_{2r+2}$, delete a matching on $t$ vertices, then add a new vertex
adjacent to the $t$ endpoints of the matching.
Let $H'_{r,t}$ be a graph on $2r+3$ vertices with 
$t+2$ vertices of degree $2r+1$ and $2r+1-t$ vertices of degree $2r$.
To form such a graph from a copy of $K_{2r+3}$, we delete the edges of a
spanning subgraph consisting of some nonnegative number of disjoint cycles
and exactly $(t+2)/2$ disjoint paths.  One (but not the only) way to form such
a subgraph, is to take the union of a near perfect matching (on $2r+2$
vertices) and a second disjoint matching on $2r+2-t$ vertices, including the
vertex missed by the first matching.

Let $\cH_r$ be the family of $(2r+1)$-regular graphs on $2k+4$ vertices
obtained from $H_{r,t}$ and $H'_{r,t}$ by adding edges from the vertex of
degree $t$ in $H_{r,t}$ to the $2r+1-t$ vertices of degree $2r$ in $H'_{r,t}$.
Since each such graph contains a cut-vertex, the family $\cH_r$ consists
entirely of $(2r+1)$-regular graphs on $2k+4$ vertices that are
non-Hamiltonian.
\end{proof}

Theorem~\ref{main1} determines a threshold for the order of a connected
$k$-regular graph that guarantees the graph is Hamiltonian.  In fact, for every
positive integer $k\ge 3$ and every even integer $n\ge 2k+4$, we can construct
connected $k$-regular graphs on $n$ vertices that are not Hamiltonian. 
Similarly, for even $k\ge 4$ and odd $n\ge 2k+3$, we can construct connected
$k$-regular graphs on $n$ vertices that are not Hamiltonian.  

Our construction is nearly identical
to that in the proof of Theorem~\ref{main1}.  If $k$ is even, we form $F_{r,t}$
starting from any $k$-regular graph on $n-k-1$ vertices (rather than 
$K_{2r+1}$).  An easy example of such graphs are circulants.  The remainder of
the construction is as before.
If $k$ is odd, we form $H'_{r,t}$ starting from any $k+1$-regular graph on
$n-k-2$ vertices (rather than $K_{2r+3}$); again circulants are an example. 
Thus we have determined exactly those orders $n$ for which a connected
$k$-regular graph on $n$ vertices must be Hamiltonian.

We may also wonder which orders $n$ guarantee that a connected $k$-regular
graph on $n$ vertices must have a Hamiltonian path.  Now we answer this
question; our proof uses Theorem~\ref{J2}. 


\begin{thm}{\rm (\cite{J2})}\label{J2}
If $G$ is 2-connected with at most $3\Delta(G)-2$ vertices,
where $\Delta(G)$ is the maximum degree of $G$,
then $G$ has a cycle containing all vertices of degree $\Delta(G)$.
\end{thm}

\begin{thm}
Every connected $k$-regular graph with at most $3k+3$ vertices
has a Hamiltonian path. 
Furthermore, we construct connected k-regular graphs on $3k+4$ vertices (when
$k\ge 6$ is even) and on $3k+5$ vertices (when $k\ge 5$ is odd) that have no
Hamiltonian path.
\end{thm}

\begin{proof}
Let $G$ be a connected $k$-regular graph with at most $3k+3$ vertices. If $G$
is 2-connected, then by Theorem~\ref{H}, $G$ has a Hamiltonian cycle, or $G \in
\{P, P'\}$.  We can easily see that the Petersen graph has a Hamiltonian path,
and every such path extends to a Hamiltonian path in $P'$.  So every
counterexample $G$ to the theorem must have a cut-vertex.

Assume that $G$ has a cut-vertex, $v$.  If $G-v$ has at least three components,
then $G$ cannot have a Hamiltonian path. But by the proof of
Theorem~\ref{main1}, each component of $G-v$ has at least $k+1$ vertices,
so $G$ has at least $3k+4$ vertices.
%
%
%
Furthermore, if $k$ is odd, then by the degree sum formula, $G$ has at least
$3k+5$ vertices.

So assume that $G-v$ has only two components, say $O_1$ and $O_2$.
Let $H_1=G[V(O_1)\cup \{v\}]$ and $H_2=G[V(O_2)\cup \{v\}]$. 
%
Note that all the vertices in $H_1$ have degree $k$ except for vertex $v$.
If $H_1$ has a cut-vertex $v_1$, then the 3 components of $G\setminus\{v,v_1\}$
have orders at least $k+1$, $k+1$, and $k$, so $G$ has order at least
$2+2(k+1)+k= 3k+4$ (and at least $3k+5$ when $k$ is odd).  Thus $H_1$ (and
similarly, $H_2$) is 2-connected.
Now if $H_1$ has at most $3k-2$ vertices,
then by Theorem~\ref{J2}, $H_1$ has a cycle containing all vertices of $H_1$
except $v$; thus $H_1$ has a Hamiltonian path with endpoint $v$.
The same is true for $H_2$.
Now since $|V(O_i)|\ge k+1$ (for both $i\in [2]$), if $G$ has at most $4k-1$
vertices, then $G$ has a Hamiltonian path, since both $H_1$ and $H_2$ have
Hamiltonian paths with endpoint $v$.  
%
%

For $k \ge 4$, we have $4k-1 \ge 3k+3$, so any $k$-regular graph with at most
$3k+3$ vertices has a Hamiltonian path. 
For $k=3$, we may assume that $|V(O_1)| \le |V(O_2)|$. 
The same argument holds as above unless $|V(G)|=3k+3=12$ and $|V(O_2)|=7$, in
which case $|V(H_2)|=8>3k-2$.  Now we have $|V(H_1)|=5$, so $d_{H_1}(v)=2$, and
thus $d_{H_2}(v)=1$.  Now we apply Theorem~\ref{J2} to $O_2$ to get a cycle
through all vertices of $O_2$ except the neighbor of $v$.  This cycle yields a
Hamiltonian path in $O_2$ that ends at a neighbor of $v$.  Thus, $G$ has a
Hamiltonian path.



For the ``Furthermore'' part, we construct a $(3k+4)$-vertex connected
$k$-regular graph without a Hamiltonian path when $k$ is even, and a
$(3k+5)$-vertex connected $k$-regular graph without a Hamiltonian path when $k$
is odd.

Let $k$ be even and at least $6$. 
Let $F_1$ be the graph obtained from $K_{k+1}$ by deleting an edge,
and let $F_2$ be the graph obtained from $K_{k+1}$ by deleting a matching on 
$k-4$ vertices.  We form a connected $k$-regular graph $F$ from two copies
of $F_1$ and one copy of $F_2$ by adding one new vertex $v$ and adding edges
from $v$ to all $k$ vertices of degree $k-1$.

Now let $k$ be odd and at least $5$. 
Let $H_1$ be $F_1$ above. 
Let $H_2$ be a graph on $k+2$ vertices, with 6 vertices of degree $k$ and $k-4$
vertices of degree $k-1$.  This is exactly $H'_{r,t}$ from the proof of
Theorem~\ref{main1}, with $t=4$ and $r=(k-1)/2$.
Now we form $H$ from two copies of $H_1$ and one copy of $H_2$ by adding one new
vertex $v$ and adding edges from $v$ to all $k$ vertices of degree $k-1$.
\end{proof}

As in the case of Theorem~\ref{main1}, for every $k\ge 3$ and $n\ge 3k+4$, we
can modify our constructions to get connected $k$-regular graphs on $n$
vertices that have no Hamiltonian path (provided that $k$ and $n$ are not both
odd).


\begin{thebibliography}{99}


%
\bibitem{H}
F. Hilbig,  Kantenstrukturen in nichthamiltonschen Graphen. Ph.D. Thesis,
Technische Universit\"{a}t Berlin (1986).

\bibitem{J1}
B. Jackson, Hamilton cycles in regular 2-connected graphs, \emph{J. Combin. Theorey B} {\bf 29} (1980), 27--46.

\bibitem{J2}
B. Jackson, Cycles through vertices of large maximum degree, \emph{J. Graph Theory} {\bf 19} (1995), 157--168.


\bibitem{OW1}
Suil O, D.B. West, Balloons, cut-edges, matchings, and total domination in regular graphs of odd degree. \emph{Journal of Graph Theory} {\bf 64} (2010) 116--131.

\bibitem{OW2}
Suil O, D.B. West, Matchings, and edge-connectivity in regular graphs.  \emph{European J. Combinatorics} {\bf 32} (2011) 324--329.


\bibitem{W}
D.B. West, \emph{Introduction to Graph Theory}, Prentice Hall, INC., Upper Sadle River, NJ, 2001.


\end{thebibliography}
\end{document}